\documentclass{article}
\usepackage{graphicx}
\usepackage{mathtools}
\usepackage{unicode-math}

\usepackage{amsmath}
\usepackage{cases}
\usepackage{comment}
\usepackage{enumitem}
\usepackage{pifont}
\usepackage{proof}
\usepackage{qtree}
\usepackage{subcaption}
\usepackage{siunitx}
\usepackage{url}
\usepackage{xcolor}
\definecolor{1f1e33}{HTML}{1F1E33}
\definecolor{mediumblue}{HTML}{0000CD}

\usepackage{tikz}
\usetikzlibrary{automata}
\usetikzlibrary{arrows.meta}
\usetikzlibrary{bending}
\usetikzlibrary{positioning}

\usepackage{amsthm}
\theoremstyle{definition}
\newtheorem{definition}{Definition}[section]
\newtheorem{example}[definition]{Example}
\newtheorem{proposition}[definition]{Proposition}
\newtheorem{lemma}[definition]{Lemma}
\newtheorem{theorem}[definition]{Theorem}
\newtheorem{corollary}[definition]{Corollary}

\usepackage[
  style=numeric, % 数字形式（plainに近い）
  backend=biber, % 内部処理エンジンをbiberに（推奨）
  maxbibnames=99, % 著者名を全員表示する
  doi=true,      % DOIを表示する
  url=true       % URLを表示する
]{biblatex}
\DeclareFieldFormat*{title}{#1\isdot} % titleを""で囲わない
\DeclareFieldFormat[book]{title}{\mkbibemph{#1\isdot}} % 本のタイトルも斜体
\renewbibmacro{in:}{
  \ifentrytype{article}
    { }% articleの場合はIn:を削除（空にする）
    {\printtext{\bibstring{in}\intitlepunct}}% それ以外（本など）は出力する
}

\usepackage{listings}
\usepackage[colorlinks=true, hidelinks]{hyperref}
\usepackage{nameref}

\newcommand{\Prop}{\symbf{Prop}}

\newcommand{\powerset}{\symcal{P}}
\renewcommand{\Box}{\symbol{"25A1}}
\newcommand{\Dia}{\symbol{"25CA}}
\newcommand{\notmodels}{\nvDash}
\newcommand{\notvdash}{\nvdash}
\newcommand{\Kmodel}{\symcal{M}}
\newcommand{\Kframe}{\symcal{F}}

\newcommand{\sfset}{\mathsf{Sub}}
\newcommand{\truthset}[1]{\lBrack #1 \rBrack}
\newcommand{\MF}{\mathbf{MF}}
\newcommand{\cl}{\mathrm{cl}}
\newcommand{\logicof}{\mathsf{Log}}

\title{Closure Semantics for Weak Modal Logics: \\ Completeness and Representations}
\author{Yuki Nishimura
    \thanks{
        Department of Computer Science, Institute of Science Tokyo,
        2-12-1 Ookayama, Meguro-ku, Tokyo, Japan.
        Email: \texttt{yuki.nishimura@comp.isct.ac.jp}.
        ORCID: \href{https://orcid.org/0009-0000-0475-1168}{0009-0000-0475-1168}.
      }
}
\date{}
\begin{document}
\maketitle

\begin{abstract}
    Weak modal logics can be obtained by restricting the axioms and inference rules governing the modal operator.
    At the extreme, one may consider the least modal logic, in which no specifically modal axiom or inference rule is imposed.
    We introduce closure semantics, a new possible-world semantics for such weak modal logics, and establish soundness and strong completeness for a range of systems.
    We also give truth-preserving translations between closure models and other possible-world semantics, including neighborhood, FMT, and Kripke semantics.
\end{abstract}

\section{Introduction}

What is the weakest modal logic?
A standard answer to this question may be the modal logic $\mathbf{K}$.
Strictly speaking, however, $\mathbf{K}$ is the weakest \emph{normal} modal logic, and weaker modal logics arise once the requirements of normality are dropped.
For example, the classical modal logic $\mathbf{E}$ is obtained by requiring only replacement of equivalents for the modal operator.
Also, the pure logic of necessitation $\mathbf{N}$ is obtained by requiring only the generalization rule (also known as the necessitation rule).
At the extreme, we arrive at a modal logic with no specifically modal axiom or inference rule, in which formulas containing the modal operator arise only through uniform substitution.
We call this least system $\mathbf{ML}_0$.
The relation of this approach to earlier studies of weak modal logics is discussed in the next subsection.

The aim of this paper is to provide a uniform possible-world semantics for this lower end of the modal hierarchy.
We introduce \emph{closure models}, in which each world $w$ is assigned a set $G(w)$ of modal formulas, called its \emph{generator}, together with an operation $\cl$ on sets of modal formulas.
In this framework, we interpret the modal clause as follows:
\[
    \Kmodel, w \models \Box \varphi \iff \varphi \in \cl(G(w)).
\]
No general closure condition, such as extensivity, monotonicity, or idempotence, is imposed on $\cl$.
Conditions on the sets $\cl(G(w))$ are instead used to capture particular modal principles.

This framework provides a uniform semantics for a range of weak modal logics.
We introduce semantic conditions corresponding to replacement of equivalents, generalization, and the axiom schemata $\mathsf{M}$, $\mathsf{C}$, $\mathsf{N}$, and $\mathsf{K}$, and prove soundness and strong completeness for every combination of these principles.
In particular, $\mathbf{ML}_0$ receives a possible-world semantics with no condition on the modal operator, while logics without modal inference rules also enjoy the finite model property and primitive-recursive decidability.
We also compare closure semantics with neighborhood, FMT, and Kripke semantics by means of truth-preserving model-to-model translations.
For each of these frameworks, we identify conditions on closure models under which translations exist in both directions.
In the Kripke case, the required model-theoretic condition is stronger than the conditions corresponding to the usual axiomatization of $\mathbf{K}$, although both determine the same modal logic.

The technical starting point of this work is the closure semantics of Kogure and Kurahashi \cite{kogure20XX}, developed for non-normal provability logics.
Their separation of a finite generator $G(w)$ from a possibly infinite closure provides an effective presentation of countermodels.
We abstract this idea from its arithmetical setting, impose no general closure-operator conditions on $\cl$, and use the resulting framework as a general semantics for weak modal logics.

\subsection{Related Work}

An early example of interest in weak modal systems appears in Lemmon's study of the Lewis systems \cite{lemmon1957}.
In particular, his system $\mathbf{S0.5}$ restricts necessitation rather than adopting the unrestricted rule familiar from normal modal logics.
Although $\mathbf{S0.5}$ is not a predecessor of the least modal logic $\mathbf{ML}_0$, Lemmon's work already illustrates an early interest in weakening the deductive principles governing the modal operator.

The least modal logic considered here has appeared under several names.
Makinson used the systems $\mathbf{S}$ and $\mathbf{L}_0$ as minimal points from which stronger modal systems could be studied \cite{makinson1971,makinson1973}, while Segerberg used $\mathbf{PC}$ as the propositional basis of his study of classical modal logic \cite{segerberg1971}.
Within Segerberg's class of classical modal logics, the weakest system is $\mathbf{E}$, since classical modal logics are required to respect replacement of equivalents.
Wansing later considered the minimal system, again under the name $\mathbf{S}$, in connection with semantics involving non-normal or impossible worlds \cite{wansing1989}.
Humberstone also discusses the system obtained by adding a modal operator to classical propositional logic without imposing any specific modal axioms or rules, regarding it as the broadest starting point for the study of modal logics \cite{humberstone2016}.
These works show that the least system has repeatedly served as the bottom of larger families of modal logics, even though it has been the primary object of semantic investigation in its own right less often.

A more recent approach is provided by non-deterministic semantics.
Gr\"atz develops non-deterministic semantic methods for modal logics and, in particular, gives a two-valued non-deterministic semantics for the least modal logic, in which the modal operator is completely unconstrained \cite{gratz2022}.
Building on this semantics, Omori and Skurt introduce a four-valued non-deterministic semantics for the same logic and prove its soundness and completeness \cite{omori2024}.
More recently, Coniglio, Paw{\l}owski, and Skurt developed restricted non-deterministic matrices (RNmatrices) as a uniform semantics for the minimal modal logic $\mathbf{M}$ and a wide range of its normal and non-normal extensions \cite{coniglio2025}.
This approach and the present one share the aim of assigning a semantics to the modal operator before stronger modal principles are imposed.
The main difference is structural: the non-deterministic approach is based on generalized truth-value assignments and does not require possible worlds, whereas closure semantics remains a possible-world semantics and assigns sets of formulas to individual worlds.
This makes it possible to compare closure models directly with other possible-world frameworks by model-to-model translations.

The pure logic of necessitation $\mathbf{N}$ was studied by Fitting, Marek, and Truszczyński, who introduced a relational semantics for this weak modal logic \cite{fitting1992}.
In their semantics, a separate accessibility relation $R_\varphi$ is associated with each modal formula $\varphi$.
This semantics validates necessitation without requiring the ordinary distribution axiom $\mathsf{K}$.
As shown below, its formula-dependent relational structure can be translated directly into closure semantics, and the exact corresponding condition on closure models is $(G)$.
Thus, closure semantics places $\mathbf{N}$ and the classical modal logic $\mathbf{E}$ within a common semantic framework, even though the two logics are incomparable.

Closure semantics is also closely related to the syntactic approach to knowledge discussed by Fagin, Halpern, Moses, and Vardi \cite{fagin1995}.
In a syntactic structure, the formulas known or believed by an agent at a state are represented directly by a set of formulas, and modal truth is determined by membership in this set.
At a purely semantic level, a closure model can similarly be collapsed by putting
\[
    H(w) = \cl(G(w))
\]
and interpreting $\Box\varphi$ by membership of $\varphi$ in $H(w)$.
The distinction between $G$ and $\cl$ is therefore not required merely to assign truth conditions to modal formulas.
It is nevertheless useful for the purposes pursued here.
On the structural side, it allows modal principles to be expressed systematically as conditions on the generated sets $\cl(G(w))$ and hence to be studied through frame definability.
On the computational side, a finite generator may provide a finite presentation of an infinite set $\cl(G(w))$, and an effective closure operation can therefore support finite countermodel constructions and decision procedures.

\subsection{Organization}
Section \ref{sec:closure-model} introduces the basic framework of closure semantics and establishes the correspondence between structural conditions on closure frames and familiar modal axioms and inference rules.
Section \ref{sec:logic} studies the modal logics determined by these conditions and establishes their soundness and strong completeness.
The discussion is divided between two subsections: logics for which the finite model property and primitive-recursive decidability are also obtained, and those for which only completeness is established here.
Section \ref{sec:representation} compares closure semantics with neighborhood, FMT, and Kripke semantics by means of truth-preserving model-to-model translations.
In the Kripke case, such a representation requires the additional condition \textsf{(KRep)}, which characterizes when a closure model can be represented by an ordinary Kripke model.
Finally, Section \ref{sec:prospects} discusses several directions for the further development of closure semantics.

\section{Closure Models}
\label{sec:closure-model}

In this section, we introduce closure models for modal logic.
Like Kripke models and neighborhood models, closure models are based on a set of possible worlds.
The additional structure consists of two components: a generator $G$, which assigns a set of modal formulas to each world, and a closure $\cl$, which maps sets of modal formulas to sets of modal formulas.
The modal operator is interpreted by membership in this generated set: a formula $\Box\varphi$ is true at a world $w$ exactly when $\varphi\in\cl(G(w))$.

We first define the language and the basic structure of closure models, and then introduce the satisfaction relation and the corresponding notions of validity and semantic consequence.

\begin{definition}
    Fix a set $\Prop$ of \emph{propositional variables}. We define modal formulas $\varphi$ inductively as follows: 
    \[
        \varphi ::= p \mid \top \mid \neg \varphi \mid \varphi \to \varphi \mid \Box \varphi,
    \]
    where $p \in \Prop$. We write $\MF$ for the set of modal formulas.
\end{definition}
We use other Boolean connectives $\bot, \land, \lor, \leftrightarrow$ as abbreviations in the usual way. Moreover, we can define a modal operator $\Dia$ as $\Dia \varphi = \neg \Box \neg \varphi$.

\begin{definition}
    A \emph{closure frame} $\Kframe$ is a tuple $(W, G, \cl)$ where:
    \begin{itemize}
        \item $W$ is a non-empty set of \emph{possible worlds}.
        \item $G: W \to \powerset(\MF)$ is a \emph{generator}.
        \item $\cl: \powerset(\MF) \to \powerset(\MF)$ is a \emph{closure}.
    \end{itemize}
    A closure model $\Kmodel$ is a tuple $(\Kframe, V)$ where $\Kframe$ is a closure frame and $V: \Prop \to \powerset(W)$ is a \emph{valuation function}.
\end{definition}
If a frame $\Kframe$ or a model $\Kmodel$ is based on a set $W$ and $w \in W$, we write $w \in \Kframe$ or $w \in \Kmodel$ instead.
We call a closure frame or a closure model \emph{finite} if $W$ is finite and $G(x)$ is finite for all $x \in W$.

We use the term \emph{closure} in a broad sense.
In contrast to the closure operators considered by Kogure and Kurahashi \cite{kogure20XX}, no conditions such as $X \subseteq \cl(X)$, monotonicity, or idempotence are imposed on $\cl$.
Thus, $\cl$ is simply an operator on sets of modal formulas, rather than a closure operator in the usual sense.

\begin{definition}
    Given a closure model $\Kmodel$, a world $w \in \Kmodel$, and $\varphi \in \MF$, we define $\Kmodel, w \models \varphi$ as follows:
    \begin{align*}
        \Kmodel, w \models \top &\text{ always} \\
        \Kmodel, w \models p &\iff w \in V(p) \text{ where } p \in \Prop \\
        \Kmodel, w \models \neg \varphi &\iff \text{not } \Kmodel, w \models \varphi \ (\Kmodel, w \notmodels \varphi) \\
        \Kmodel, w \models \varphi \to \psi &\iff \Kmodel, w \notmodels \varphi \text{ or } \Kmodel, w \models \psi \\
        \Kmodel, w \models \Box \varphi &\iff \varphi \in \cl(G(w)).
    \end{align*}
    If $\Kmodel$ under consideration is clear from the context, we simply write $w \models \varphi$.
\end{definition}
For a set $\Gamma$ of formulas, if we have $w \models \gamma$ for all $\gamma \in \Gamma$, we simply write $w \models \Gamma$.

\begin{definition}
    We define the following notions of validity and semantic consequence. Here, $\varphi$ is a modal formula and $\Gamma$ is a set of formulas:
    \begin{itemize}
        \item A formula $\varphi$ is \emph{valid in a model $\Kmodel$}, denoted by $\Kmodel \models \varphi$, if $\Kmodel, w \models \varphi$ for all $w \in \Kmodel$. Moreover, a formula $\varphi$ is \emph{valid in a frame $\Kframe$}, denoted by $\Kframe \models \varphi$, if $(\Kframe, V), w \models \varphi$ for all $V$ and $w \in \Kframe$.
        \item A formula $\varphi$ is \emph{valid}, denoted by $\models \varphi$, if $\Kmodel, w \models \varphi$ for all $\Kmodel$ and $w \in \Kmodel$. 
        \item \emph{A set $\Gamma$ semantically entails $\varphi$}, denoted by $\Gamma \models \varphi$, if for all $\Kmodel$ and $w \in \Kmodel$, $\Kmodel, w \models \Gamma$ implies $\Kmodel, w \models \varphi$.
        \item Given a class $\mathbb{C}$ of closure frames, \emph{$\Gamma$ semantically entails $\varphi$ over $\mathbb{C}$}, denoted by $\Gamma \models_{\mathbb{C}} \varphi$, if for any $\symcal{F} \in \mathbb{C}$, $V$, and $w \in \Kframe$ $(\symcal{F}, V), w \models \Gamma$ implies $(\symcal{F}, V), w \models \varphi$.
    \end{itemize}
\end{definition}
We write $\logicof(\Kframe)$ to mean the set of formulas $\varphi$ such that $\Kframe \models \varphi$ and $\logicof(\Kmodel)$ to mean the set of formulas $\varphi$ such that $\Kmodel \models \varphi$.

\subsection{Frame Definability}
\label{subsec:frame-definability}
One advantage of closure semantics is that modal principles can be represented as structural properties of the sets $\cl(G(w))$ associated with worlds.
In this subsection, we introduce conditions corresponding to the inference rules \textsf{RE} and \textsf{Gen} and to the axiom schemata \textsf{M, C, N,} and \textsf{K}.
We then verify that these semantic conditions are exactly characterized by the corresponding rules and axiom schemata.

\begin{definition}
    We define the properties \textsf{(E)}, \textsf{(G)}, \textsf{(M)}, \textsf{(C)}, \textsf{(N)}, \textsf{(K)} of closure frames in Figure \ref{fig:property-closure-frame}.
\end{definition}
All the properties introduced above are formulated as properties of closure frames.
This follows the usual treatment of definability in Kripke and neighborhood semantics, where modal axioms and inference rules are associated with frame properties rather than with particular valuations.
Accordingly, the correspondence result below is stated in terms of validity on closure frames.

\begin{figure}
    For any $w \in \Kframe$ and $\varphi, \psi \in \MF$,
    \begin{description}
        \item[\textsf{(E)}] $\Kframe \models \varphi \leftrightarrow \psi \implies (\varphi \in \cl(G(w)) \iff \psi \in \cl(G(w)))$
        \item[\textsf{(G)}\footnotemark] $\Kframe \models \varphi \implies \varphi \in \cl(G(w))$
        \item[\textsf{(M)}] $\varphi \land \psi \in \cl(G(w)) \implies \varphi, \psi \in \cl(G(w))$
        \item[\textsf{(C)}] $\varphi, \psi \in \cl(G(w)) \implies \varphi \land \psi \in \cl(G(w))$
        \item[\textsf{(N)}] $\top \in \cl(G(w))$
        \item[\textsf{(K)}] $\varphi, \varphi \to \psi \in \cl(G(w)) \implies \psi \in \cl(G(w))$
    \end{description}
    \caption{Properties of closure frames}
    \label{fig:property-closure-frame}
\end{figure}
\footnotetext{The letter G stands for the \emph{generalization rule}, usually called the \emph{necessitation rule}. We use \textsf{(G)} to avoid confusion with the property corresponding to the axiom scheme\textsf{N}.}

Figure \ref{fig:axiom-and-rules} lists the corresponding axiom schemata and inference rules.
The frame properties above are characterized by these principles.

\begin{proposition}
    Let $\Kframe$ be a closure frame. Then, the following statements hold:
    \begin{enumerate}[label=(\alph*)]
        \item $\Kframe$ satisfies \textsf{(E)} if and only if $\logicof(\Kframe)$ is closed under \textsf{RE}.
        \item $\Kframe$ satisfies \textsf{(G)} if and only if $\logicof(\Kframe)$ is closed under \textsf{Gen}.
        \item For each $\mathsf{X} \in \{ \mathsf{M}, \mathsf{C}, \mathsf{N}, \mathsf{K} \}$, $\Kframe$ satisfies \textsf{(X)} if and only if $\mathsf{X} \in \logicof(\Kframe)$.
    \end{enumerate}
    \label{prop:correspondence}
\end{proposition}
\begin{proof}
    We prove only the cases for \textsf{(E)} and \textsf{(M)}. The other cases are proved similarly.

    Suppose first that $\Kframe$ satisfies \textsf{(E)}, and assume that $\Kframe \models \varphi \leftrightarrow \psi$.
    Then, for every $w \in \Kframe$, $\varphi \in \cl(G(w))$ if and only if $\psi \in \cl(G(w))$ by \textsf{(E)}.
    Hence, for every valuation $V$ and every $w \in \Kframe$,
    \[
        (\Kframe, V), w \models \Box \varphi \iff (\Kframe, V), w \models \Box \psi.
    \]
    Therefore, $\Kframe \models \Box \varphi \leftrightarrow \Box \psi$.

    Conversely, suppose that $\logicof(\Kframe)$ is closed under \textsf{RE}, and assume that $\Kframe \models \varphi \leftrightarrow \psi$.
    From the latter, we have $\varphi \leftrightarrow \psi \in \logicof(\Kframe)$, and hence $\Box \varphi \leftrightarrow \Box \psi \in \logicof(\Kframe)$ by \textsf{RE}.
    Therefore, for every $w \in \Kframe$, $\varphi \in \cl(G(w))$ if and only if $\psi\in\cl(G(w))$.

    Next, suppose that $\Kframe$ satisfies \textsf{(M)}.
    Let $V$ be any valuation and $w \in \Kframe$, and suppose that $(\Kframe, V), w \models \Box (\varphi \land \psi)$.
    Then, $\varphi \land \psi \in \cl(G(w))$ holds.
    By \textsf{(M)}, we have $\varphi, \psi \in \cl(G(w))$, and hence
    $(\Kframe, V), w \models \Box \varphi \land \Box \psi$.

    Conversely, suppose that $\mathsf{M} \in \logicof(\Kframe)$.
    Take any $w \in \Kframe$, and suppose that $\varphi \land \psi \in \cl(G(w))$.
    In this situation, we have $(\Kframe, V), w \models \Box (\varphi \land \psi)$ for any valuation $V$. Then, it follows that $(\Kframe, V), w \models \Box \varphi \land \Box \psi$ by \textsf{M}.
    Therefore, $\varphi, \psi \in \cl(G(w))$.
\end{proof}

\begin{figure}
    \noindent \underline{Axiom Schemata}
    \begin{description}
        \item[\textsf{M}] $\Box (\varphi \land \psi) \to \Box \varphi \land \Box \psi$
        \item[\textsf{C}] $\Box \varphi \land \Box \psi \to \Box (\varphi \land \psi)$
        \item[\textsf{N}] $\Box \top$
        \item[\textsf{K}] $\Box (\varphi \to \psi) \to (\Box \varphi \to \Box \psi)$ 
    \end{description}

    \noindent \underline{Inference Rules}
    \begin{description}
        \item[\textsf{RE}] From $\varphi \leftrightarrow \psi$, infer $\Box \varphi \leftrightarrow \Box \psi$.
        \item[\textsf{Gen}] From $\varphi$, infer $\Box \varphi$.
    \end{description}
    \caption{Modal axiom schemata and inference rules}
    \label{fig:axiom-and-rules}
\end{figure}

For $\Lambda \subseteq \{ \mathsf{E, G, M, C, N, K} \}$, $\mathbf{CF}_\Lambda$ denotes the class of frames satisfying all conditions \textsf{(X)} such that $\mathsf{X} \in \Lambda$. In particular, $\mathbf{CF} = \mathbf{CF}_\emptyset$ is the class of all closure frames. We abbreviate $\Gamma \models_{\mathbf{CF}_\Lambda} \varphi$ to $\Gamma \models_\Lambda \varphi$.

\section{Modal Logics}
\label{sec:logic}
Having introduced closure semantics, we now define the modal logics to be studied with respect to this semantics.
As our weakest starting point, we take classical propositional logic formulated in a modal language, without imposing any specifically modal axioms or inference rules.
We then obtain stronger modal logics by adding some of the axiom schemata \textsf{M}, \textsf{C}, \textsf{N}, and \textsf{K} and the inference rules \textsf{RE} and \textsf{Gen} introduced in the previous section.

The soundness of these logics follows directly from the corresponding semantic conditions established above.
The main task of this section is therefore completeness, for which we construct suitable countermodels for formulas that are not derivable.
We first treat logics without modal inference rules, where a one-point construction suffices, and then turn to the general case involving \textsf{RE} or \textsf{Gen}.

\begin{definition}
    A \emph{modal logic} $L$ is a subset of $\MF$ that contains all classical tautologies and is closed under \textit{modus ponens} and uniform substitution. We say that a formula $\varphi$ is \emph{provable} in a logic $L$ from $\Gamma$ if there are formulas $\gamma_1, \ldots, \gamma_n \in \Gamma$ such that $\gamma_1 \land \cdots \land \gamma_n \to \varphi \in L$\footnote{We allow $n = 0$, in which case the empty conjunction is understood as $\top$.}.
\end{definition}

\begin{definition}
    For $\Lambda \subseteq \{ \mathsf{E, G, M, C, N, K} \}$, we define $L_\Lambda$ as the smallest modal logic satisfying the following conditions:
    \begin{itemize}
        \item $L_\Lambda$ is closed under \textsf{RE} if $\mathsf{E} \in \Lambda$.
        \item $L_\Lambda$ is closed under \textsf{Gen} if $\mathsf{G} \in \Lambda$.
        \item For each $\mathsf{X} \in \{ \mathsf{M, C, N, K} \}$, $L_\Lambda$ contains an axiom scheme\textsf{X} if $\mathsf{X} \in \Lambda$.
    \end{itemize}
    We write $\Gamma \vdash_\Lambda \varphi$ to mean that $\varphi$ is provable in $L_\Lambda$ from $\Gamma$.
\end{definition}
In particular, we use the following notation: $\mathbf{ML}_0 = L_\emptyset, \mathbf{E} = L_{\mathsf{E}}, \mathbf{N} = L_\mathsf{G}, \mathbf{K} = L_\mathsf{GK}$.

We write $\mathbf{ML}_0$ for the weakest modal logic, namely, classical propositional logic formulated in the modal language with no additional modal axiom or inference rule.
The subscript $0$ indicates that no specifically modal principle is imposed.
The same minimal system has appeared under various names in the literature, including $\mathbf{S}$ in Makinson \cite{makinson1971} and Wansing \cite{wansing1989}, $\mathbf{L}_0$ in Makinson \cite{makinson1973}, $\mathbf{PC}$ in Segerberg \cite{segerberg1971}, and $\mathbf{M}$ in Omori and Skurt \cite{omori2024}.

Observe that soundness follows immediately from Proposition \ref{prop:correspondence}.

\begin{theorem}
    Let $\Lambda \subseteq \{ \mathsf{E, G, M, C, N, K} \}$. Then, the logic $L_\Lambda$ is sound with respect to the class $\mathbf{CF}_\Lambda$ of frames. That is, $\Gamma \vdash_\Lambda \varphi$ implies $\Gamma \models_\Lambda \varphi$.
    \label{thm:soundness-closure-frame}
\end{theorem}

As a simple application of closure semantics, we can separate the logics $\mathbf{E}$ and $\mathbf{N}$ by suitable closure models. We write $\mathrm{id}$ for the identity closure, defined by $\mathrm{id}(\Gamma) = \Gamma$ for every $\Gamma \subseteq \MF$.

\begin{proposition}
The modal logics $\mathbf{E}$ and $\mathbf{N}$ are incomparable; that is, neither $\mathbf{E} \subseteq \mathbf{N}$ nor $\mathbf{N} \subseteq \mathbf{E}$ holds.
\label{prop:E-N-incomparable}
\end{proposition}

\begin{proof}
By Theorem \ref{thm:soundness-closure-frame}, it suffices to construct a closure model whose underlying frame satisfies \textsf{(G)} in which some theorem of $\mathbf{E}$ is false, and a closure model whose underlying frame satisfies \textsf{(E)} in which some theorem of $\mathbf{N}$ is false.

For the first direction, let $\Kmodel_1 = (\{w\}, G, \mathrm{id}, V)$ be a closure model such that $G(w) = \MF \setminus \{p\}$.
Then, its underlying frame satisfies \textsf{(G)}; indeed, the only formula not belonging to $G(w)$ is $p$, which is not valid on the frame.
Since $p\leftrightarrow\neg\neg p$ is a classical tautology,
$\Box p \leftrightarrow \Box \neg \neg p \in \mathbf{E}$.
However, $\Kmodel_1, w \notmodels \Box p$ and $\Kmodel_1, w \models \Box \neg \neg p$.
Therefore, $\mathbf{E} \nsubseteq \mathbf{N}$.

For the converse direction, let $\Kmodel_2 = (\{w\}, G, \mathrm{id}, V)$ be a closure model such that $G(w) = \emptyset$.
Its underlying frame satisfies \textsf{(E)} trivially.
Since $\top \in \mathbf{N}$, we have $\Box \top \in \mathbf{N}$ by \textsf{Gen}.
However, $\Kmodel_2, w \notmodels \Box \top$.
Therefore, $\mathbf{N}\nsubseteq\mathbf{E}$.
\end{proof}

\subsection{Completeness without Inference Rules}
\label{subsec:completeness-onepoint}
We first consider the case where neither \textsf{E} nor \textsf{G} belongs to $\Lambda$, that is, where the logic contains no inference rule governing the modal operator.
In this case, completeness can be proved by constructing a finite one-point canonical model.

More precisely, starting from a finite set of formulas, we construct a finite generator and define the corresponding closure $\cl_\Lambda$ effectively.
This yields a finite countermodel for any formula that is not derivable.
In the latter part of this subsection, we use the finitary and effective nature of this construction to establish the finite model property and decidability.

\begin{definition}
    Let $\varphi \in \MF$. We define $\sfset(\varphi)$ as the set of all subformulas of $\varphi$. Moreover, we define some sets of formulas as follows:
    \begin{align*}
        \sfset^\Box (\varphi) &= \sfset(\varphi) \cup \{ \Box \psi \mid \psi \in \sfset(\varphi) \} \cup \{ \top, \Box \top \} \\
        \sfset^* (\varphi) &= \sfset^\Box (\varphi) \cup \{ \sim \psi \mid \psi \in \sfset^\Box (\varphi) \}.
    \end{align*}
    Here, $\sim \psi$ denotes $\chi$ if $\psi$ is of the form $\neg \chi$, and denotes $\neg \psi$ otherwise.

    We also define $\sfset(\Gamma)$ for a set $\Gamma$ of modal formulas as the union of $\sfset(\gamma)$ for every $\gamma \in \Gamma$. The sets $\sfset^\Box(\Gamma)$ and $\sfset^*(\Gamma)$ are defined likewise.
\end{definition}

\begin{definition}
    Let $\Gamma, \Delta$ be sets of modal formulas.
    \begin{itemize}
        \item $\Gamma$ is \emph{$L_\Lambda$-consistent} if $\Gamma \notvdash_\Lambda \bot$.
        \item $\Gamma$ is \emph{maximally $L_\Lambda$-consistent} if $\Gamma$ is $L_\Lambda$-consistent and $\psi \in \Gamma$ or $\sim \psi \in \Gamma$ for all $\psi \in \MF$.
        \item $\Gamma$ is \emph{$\sfset^*(\Delta)$-maximally $L_\Lambda$-consistent} if $\Gamma$ is $L_\Lambda$-consistent and $\psi \in \Gamma$ or $\sim \psi \in \Gamma$ for all $\psi \in \sfset^*(\Delta)$.
    \end{itemize}
\end{definition}

We use the following standard form of Lindenbaum's lemma: if $X$ is $L_\Lambda$-consistent, it can be extended to a ($\sfset^*(\Gamma)$-)maximally $L_\Lambda$-consistent set.
The proof is routine and is omitted.

We define a closure $\cl_\Lambda$. The idea is to build the semantic conditions corresponding to the modal axioms in $\Lambda$ directly into the generation process.
\begin{definition}
    We define $\cl_\Lambda$. First, for a given set $\Gamma$ of formulas, we define a sequence $(\Gamma^\Lambda_n)_{n \in \omega}$ of sets as follows:
    \begin{enumerate}[label=\arabic*.]
        \item $\Gamma^\Lambda_0 = \Gamma \cup \{ \top \}$ if $\mathsf{N} \in \Lambda$, and $\Gamma^\Lambda_0 = \Gamma$ otherwise.
        \item $\Gamma^\Lambda_{n+1}$ is the union of $\Gamma^\Lambda_n$ and the following sets:
        \begin{enumerate}[label=(\alph*)]
            \item $\{ \psi, \chi \mid \psi \land \chi \in \Gamma^\Lambda_n \}$ if $\mathsf{M} \in \Lambda$ \label{closure-construction-case-m}
            \item $\{ \psi \land \chi \mid \psi, \chi \in \Gamma^\Lambda_n \}$ if $\mathsf{C} \in \Lambda$ \label{closure-construction-case-c}
            \item $\{ \chi \mid \psi, \psi \to \chi \in \Gamma^\Lambda_n \}$ if $\mathsf{K} \in \Lambda$. \label{closure-construction-case-k}
        \end{enumerate}
    \end{enumerate}
    Then, we define $\cl_\Lambda$ as $\cl_\Lambda(\Gamma) = \bigcup_{n \in \omega} \Gamma^\Lambda_n$.
    \label{def:closure-onepoint}
\end{definition}

We now use this closure to construct the countermodel. Since no modal inference rule is present, we only need a single world. We take a maximally consistent set over the relevant finite fragment and use the formulas occurring under boxes in this set as the generator. The closure $\cl_\Lambda$ supplies exactly the additional modal information required by the axioms in $\Lambda$.
\begin{definition}
    Given an $L_\Lambda$-consistent set $\Gamma$ of formulas, we define $\Kmodel^\Gamma_\Lambda = (W, G, \cl, V)$ as follows:
    \begin{itemize}
        \item $W$ is a singleton $\{ \Delta_\Gamma \}$, where $\Delta_\Gamma$ is a $\sfset^*(\Gamma)$-maximally $L_\Lambda$-consistent set satisfying $\Gamma \subseteq \Delta_\Gamma \subseteq \sfset^*(\Gamma)$.
        \item $G(\Delta_\Gamma) = \{ \varphi \in \sfset^*(\Gamma) \mid \Box \varphi \in \Delta_\Gamma \}$.
        \item $\cl = \cl_\Lambda$ defined in Definition \ref{def:closure-onepoint}.
        \item $V(p) = \{ \Delta_\Gamma \}$ if $p \in \Delta_\Gamma$, and $V(p) = \emptyset$ otherwise.
    \end{itemize}
    \label{def:countermodel-onepoint}
\end{definition}

We next prove the truth lemma: a formula is true in the world $\Delta_\Gamma$ if it is in $\Delta_\Gamma$.
To prove it, the nontrivial point is to relate membership in the generated closure to membership of the corresponding boxed formula in the maximally consistent set. The following lemma provides exactly this direction.
\begin{lemma}
    For any $\Box \varphi \in \sfset^*(\Gamma)$, $\varphi \in \cl_\Lambda(G(\Delta_\Gamma))$ implies $\Box \varphi \in \Delta_\Gamma$.
    \label{lem:reflection-onepoint}
\end{lemma}
\begin{proof}
    We show the following claim by induction on $n$:
    \begin{center}
        \textbf{Claim.} for any $\varphi$, $\varphi \in (G(\Delta_\Gamma))^\Lambda_n$ implies $\Delta_\Gamma \vdash_\Lambda \Box \varphi$.
    \end{center}
    \begin{proof}[Proof of the Claim.]
        In the base case $n = 0$, it is straightforward from the definition of $G(\Delta_\Gamma)$, except when $\varphi = \top$ and $\top \notin G(\Delta_\Gamma)$. In this case, $\top$ is added to $(G(\Delta_\Gamma))^\Lambda_0$ because $\mathsf{N} \in \Lambda$. Hence, $\Delta_\Gamma \vdash_\Lambda \Box \top$ follows from axiom \textsf{N}.
    
        For the inductive step, suppose that the claim holds for $n \leq k$ and $\varphi \in (G(\Delta_\Gamma))^\Lambda_{k+1}$.
        If $\varphi \in (G(\Delta_\Gamma))^\Lambda_k$, the claim follows immediately from the induction hypothesis.
    
        Suppose that $\varphi$ is added by clause \ref{closure-construction-case-m}, so that $\varphi \land \psi \in (G(\Delta_\Gamma))^\Lambda_k$ for some $\psi$.
        By the induction hypothesis, $\Delta_\Gamma \vdash_\Lambda \Box(\varphi \land \psi)$.
        Hence, $\Delta_\Gamma \vdash_\Lambda \Box\varphi$ by the axiom scheme \textsf{M}.
        The cases for clauses \ref{closure-construction-case-c} and \ref{closure-construction-case-k} are proved similarly, using the axiom schemata \textsf{C} and \textsf{K}, respectively.
    \end{proof}
    Finally, suppose that $\Box \varphi \in \sfset^*(\Gamma)$ and $\varphi \in \cl_\Lambda(G(\Delta_\Gamma))$. Then, the claim yields that $\Delta_\Gamma \vdash_\Lambda \Box \varphi$. Since $\Delta_\Gamma$ is $\sfset^*(\Gamma)$-maximally $L_\Lambda$-consistent, we have $\Box \varphi \in \Delta_\Gamma$.
\end{proof}

\begin{lemma}
    For any $\varphi \in \sfset^*(\Gamma)$, $\Kmodel_\Lambda^\Gamma, \Delta_\Gamma \models \varphi$ if and only if $\varphi \in \Delta_\Gamma$.
    \label{lem:truth-lemma-onepoint}
\end{lemma}
\begin{proof}
    By induction on the complexity of $\varphi$.
    Since the base case and the inductive cases for the Boolean connectives are straightforward, we only show the case $\varphi = \Box \psi$.

    If $\Delta_\Gamma \models \Box \psi$, then $\psi \in \cl_\Lambda(G(\Delta_\Gamma))$ by the satisfaction clause for $\Box$. By Lemma \ref{lem:reflection-onepoint}, $\Box \psi \in \Delta_\Gamma$. Conversely, if $\Box \psi \in \Delta_\Gamma$, then $\psi \in G(\Delta_\Gamma) \subseteq \cl_\Lambda(G(\Delta_\Gamma))$. Therefore, we have $\Delta_\Gamma \models \Box \psi$.
\end{proof}

\begin{lemma}
    The model $\Kmodel_\Lambda^\Gamma$ satisfies the property \textsf{(X)} if $\mathsf{X} \in \Lambda$.
    \label{lem:corresponding-lemma-onepoint}
\end{lemma}
\begin{proof}
    If $\mathsf{N} \in \Lambda$, then $\top \in \cl_\Lambda(G(\Delta_\Gamma))$ is straightforward from the construction of $\cl_\Lambda$. 
    
    Suppose that $\mathsf{M} \in \Lambda$ and $\varphi \land \psi \in \cl_\Lambda(G(\Delta_\Gamma))$.
    Then, by clause \ref{closure-construction-case-m} of Definition \ref{def:closure-onepoint}, we have $\varphi, \psi \in \cl_\Lambda(G(\Delta_\Gamma))$. We can prove the cases for \textsf{C} and \textsf{K} similarly.
\end{proof}

\begin{theorem}
    Let $\Lambda \subseteq \{ \mathsf{M, C, N, K} \}$. Then, the logic $L_\Lambda$ is strongly complete with respect to the class $\mathbf{CF}_\Lambda$ of frames. That is, $\Gamma \models_\Lambda \varphi$ implies $\Gamma \vdash_\Lambda \varphi$.
    \label{thm:completeness-onepoint}
\end{theorem}
\begin{proof}
    Suppose that $\Gamma \notvdash_\Lambda \varphi$. Then, the set $\Sigma = \Gamma \cup \{ \neg \varphi \}$ is $L_\Lambda$-consistent. By Lindenbaum's Lemma, we can make a $\sfset^*(\Sigma)$-maximally $L_\Lambda$-consistent set $\Delta_\Sigma$ such that $\Gamma \cup \{ \neg \varphi \} \subseteq \Delta_\Sigma$. Then, by Lemma \ref{lem:truth-lemma-onepoint}, we have a model $\Kmodel_\Lambda^\Sigma$ and its world $\Delta_\Sigma$ such that $\Delta_\Sigma \models \Gamma$ and $\Delta_\Sigma \notmodels \varphi$. Moreover, $\Kmodel_\Lambda^\Sigma$ satisfies all properties \textsf{(X)} such that $\mathsf{X} \in \Lambda$ by Lemma \ref{lem:corresponding-lemma-onepoint}. Therefore, we have $\Gamma \notmodels_\Lambda \varphi$.
\end{proof}
In particular, we have that $\mathbf{ML}_0$ is sound and complete with respect to the class of closure frames $\mathbf{CF}$.

In addition to completeness, we establish the finite model property and decidability of the logics considered above.
We formulate both properties with respect to a fixed modal formula $\varphi$, not to a set $\Gamma$ of formulas. 
That is, the finite model property provides a finite countermodel whenever $\notvdash_\Lambda \varphi$, and decidability amounts to determining in finite time whether $\vdash_\Lambda \varphi$.
The former follows immediately from the preceding completeness construction, while the latter follows from its effective character.

\begin{corollary}
    Let $\Lambda \subseteq \{ \mathsf{M, C, N, K} \}$. Then, $L_\Lambda$ has the finite model property.
    \label{cor:fmp}
\end{corollary}
\begin{proof}
    Suppose that $\notvdash_\Lambda \varphi$. In Definition \ref{def:countermodel-onepoint}, assume that $\Gamma = \{ \neg \varphi \}$.
    Then, the model $\Kmodel^\Gamma_\Lambda$ is a countermodel for $\varphi$.
    Moreover, $W$ is finite and $G(\Delta_\Gamma)$ is also finite since $\sfset^*(\neg \varphi)$ is finite.
\end{proof}

\begin{corollary}
    Let $\Lambda \subseteq \{ \mathsf{M, C, N, K} \}$. Then $L_\Lambda$ is decidable.
    In fact, $L_\Lambda$ is primitive recursive.
    \label{cor:decidability}
\end{corollary}
\begin{proof}
    Fix a modal formula $\varphi$.
    Since $\sfset^*(\neg \varphi)$ is finite, there are only finitely many subsets $\Delta \subseteq \sfset^*(\neg \varphi)$.
    For each such $\Delta$, we construct the one-point model $\Kmodel = (W, G, \cl, V)$ as follows:
    \begin{itemize}
        \item $W = \{ \Delta \}$
        \item $G(\Delta) = \{ \psi \in \sfset^*(\neg \varphi) \mid \Box \psi \in \Delta \}$
        \item $\cl = \cl_\Lambda$
        \item $V(p) = \{ \Delta \}$ if $p \in \Delta$, and $V(p) = \emptyset$ otherwise.
    \end{itemize}
    These are finite models, and by the definition of $\cl_\Lambda$, each such model satisfies all the conditions \textsf{(X)} with $\mathsf{X} \in \Lambda$. Moreover, the relation $\psi \in \cl_\Lambda(G(\Delta))$ is primitive recursive.
    Indeed, for a fixed $\psi$, it suffices to perform the construction of Definition \ref{def:closure-onepoint} within the finite set of subformulas of $G(\Delta)\cup\{\psi,\top\}$.
    This restricted construction stabilizes after a primitively recursively bounded number of stages, and each stage is primitive recursive.
    Hence, for each candidate model, whether $\neg \varphi$ is true in it can be determined in finite time.
    
    If $\notvdash_\Lambda \varphi$, the countermodel constructed in Corollary \ref{cor:fmp} yields one of these candidates in which $\varphi$ is false.
    Conversely, if one of these candidates falsifies $\varphi$, then $\notvdash_\Lambda \varphi$ by soundness.
    Therefore, $\vdash_\Lambda \varphi$ if and only if $\varphi$ is true in all of the finitely many candidates above.
\end{proof}

From a purely semantic point of view, one could replace the pair $(G, \cl)$ with a single function $H: W \to \mathcal{P}(\MF)$ by putting $H(w) = \cl(G(w))$, and define $w \models \Box \varphi$ simply by $\varphi \in H(w)$.
Thus, the distinction between the generator $G$ and the closure $\cl$ is not essential for interpreting modal formulas.
However, this distinction becomes useful when the finite model property and decidability are considered.
The set $\cl(G(w))$ may be infinite even when the generator $G(w)$ is finite, so the pair $(G, \cl)$ provides a finite presentation of potentially infinite modal information.
Moreover, when membership in $\cl(G(w))$ is effectively decidable, as in the constructions above, this finite presentation makes it possible to search effectively through the relevant countermodels.
In this sense, the usefulness of closure semantics lies not only in its semantic representation of modal operators, but also in the separation of finite generators from an effectively generated closure.

\subsection{Completeness with Inference Rules}
\label{subsec:completeness-full}

If either \textsf{E} or \textsf{G} belongs to $\Lambda$, the finite one-point construction from the previous subsection is no longer sufficient.
Unlike the axiom schemata treated there, the rules \textsf{RE} and \textsf{Gen} involve validity and cannot be handled solely within a single finite fragment.
We therefore pass to the full canonical model whose worlds are all maximally $L_\Lambda$-consistent sets.

\begin{definition}
    We define $\cl_\Lambda$ for $\Lambda \subseteq \{ \mathsf{E, G, M, C, N, K} \}$. First, for a given set $\Gamma$ of formulas, we define a sequence $(\Gamma^\Lambda_n)_{n \in \omega}$ of sets as follows:
    \begin{enumerate}[label=\arabic*.]
        \item $\Gamma^\Lambda_0$ is the union of $\Gamma$ and the following sets:
        \begin{enumerate}[label=(\alph*)]
            \item $L_\Lambda$ if $\mathsf{G} \in \Lambda$
            \item $\{ \top \}$ if $\mathsf{N} \in \Lambda$.
        \end{enumerate}
        \item $\Gamma^\Lambda_{n+1}$ is defined in the same way as in Definition \ref{def:closure-onepoint}.
    \end{enumerate}
    Then, we define $\cl_\Lambda$ as $\cl_\Lambda(\Gamma) = \bigcup_{n \in \omega} \Gamma^\Lambda_n$.
    \label{def:closure-full}
\end{definition}

\begin{definition}
    We define $\Kmodel_\Lambda = (W, G, \cl, V)$ as follows:
    \begin{itemize}
        \item $W$ is the set of all maximally $L_\Lambda$-consistent sets.
        \item $G(\Gamma) = \{ \varphi \in \MF \mid \Box \varphi \in \Gamma \}$.
        \item $\cl = \cl_\Lambda$ defined in Definition \ref{def:closure-full}.
        \item $V(p) = \{ \Gamma \mid p \in \Gamma \}$.
    \end{itemize}
    \label{def:countermodel-full}
\end{definition}
Observe that the model $\Kmodel_\Lambda$ defined above is no longer finite, since its worlds are all maximally $L_\Lambda$-consistent sets rather than maximally consistent sets over a fixed finite fragment.

\begin{lemma}
    For any $\Gamma \in W$ and $\varphi \in \MF$, if $\varphi \in \cl(G(\Gamma))$, then $\Box \varphi \in \Gamma$.
    \label{lem:reflection-full}
\end{lemma}
\begin{proof}
    The proof proceeds as in Lemma \ref{lem:reflection-onepoint}.
    
    The only additional case is that $\varphi \in (G(\Gamma))^\Lambda_0$ because $\varphi \in L_\Lambda$.
    In this case, we have $\mathsf{G} \in \Lambda$, so $\Box\varphi \in L_\Lambda$ by \textsf{Gen}, and hence $\Gamma \vdash_\Lambda \Box\varphi$.
\end{proof}

\begin{lemma}
    For any $\Gamma \in W$ and any $\varphi \in \MF$, $\Kmodel_\Lambda, \Gamma \models \varphi$ if and only if $\varphi \in \Gamma$.
    \label{lem:truth-lemma-full}
\end{lemma}
\begin{proof}
    By induction on the complexity of $\varphi$.
    Particularly, the case $\varphi = \Box \psi$ is proved as in Lemma \ref{lem:truth-lemma-onepoint}, using Lemma \ref{lem:reflection-full}.
\end{proof}

\begin{lemma}
    The model $\Kmodel_\Lambda$ satisfies the property \textsf{(X)} if $\mathsf{X} \in \Lambda$.
    \label{lem:corresponding-lemma-full}
\end{lemma}
\begin{proof}
    We only prove the cases $\mathsf{E} \in \Lambda$ and $\mathsf{G} \in \Lambda$.
    The other cases are proved as in Lemma \ref{lem:corresponding-lemma-onepoint}. 
    Let $\Kframe_\Lambda$ be the underlying frame of $\Kmodel_\Lambda$.

    To prove the case $\mathsf{E} \in \Lambda$, suppose $\Kframe_\Lambda \models \varphi \leftrightarrow \psi$. In particular, we have $\Kmodel_\Lambda \models \varphi \leftrightarrow \psi$. This means that $\Gamma \models \varphi \leftrightarrow \psi$ for all $\Gamma \in W$, and that is equivalent to $\varphi \leftrightarrow \psi \in \Gamma$. Since $\bigcap_{\Gamma \in W} \Gamma = L_\Lambda$, we have $\varphi \leftrightarrow \psi \in L_\Lambda$. By \textsf{RE}, $\Box \varphi \leftrightarrow \Box \psi \in L_\Lambda$. Thus, $\Box \varphi \in \Gamma$ is equivalent to $\Box \psi \in \Gamma$. By the definition of $G$ and Lemma \ref{lem:reflection-full}, we have $\chi \in \cl_\Lambda(G(\Gamma))$ if and only if $\Box \chi \in \Gamma$ for all $\chi \in \MF$, and hence $\varphi \in \cl_\Lambda(G(\Gamma))$ if and only if $\psi \in \cl_\Lambda(G(\Gamma))$.

    Suppose now that $\mathsf{G} \in \Lambda$, suppose $\Kframe_\Lambda \models \varphi$, which implies we have $\Kmodel_\Lambda \models \varphi$. By the same argument as the previous case, we have $\varphi \in L_\Lambda$. By \textsf{Gen}, $\Box \varphi \in L_\Lambda$. Since $L_\Lambda \subseteq \Gamma$ for any $\Gamma$, we have $\Box \varphi \in \Gamma$. By the definition of $G$, we have $\varphi \in G(\Gamma)$, which yields $\varphi \in \cl_\Lambda(G(\Gamma))$ as in the previous case.
\end{proof}

\begin{theorem}
    Let $\Lambda \subseteq \{ \mathsf{E, G, M, C, N, K} \}$. Then, the logic $L_\Lambda$ is strongly complete with respect to the class $\mathbf{CF}_\Lambda$ of frames. That is, $\Gamma \models_\Lambda \varphi$ implies $\Gamma \vdash_\Lambda \varphi$.
\end{theorem}
\begin{proof}
    Similar argument as that for Theorem \ref{thm:completeness-onepoint}, using Lemmas \ref{lem:truth-lemma-full} and \ref{lem:corresponding-lemma-full}.
\end{proof}

\section{Representation Theorems}
\label{sec:representation}
In this section, we compare closure semantics with three existing semantic frameworks for modal logic: neighborhood semantics, FMT semantics for the modal logic $\mathbf{N}$, and ordinary Kripke semantics.
All of these semantics are based on a set of possible worlds equipped with additional structure, although the nature of that structure differs from one framework to another.
This common format allows us to compare the semantics directly at the level of models rather than only through the logics they determine.
Accordingly, we construct model-to-model translations and show that they preserve the truth of every modal formula.

Since the results of this section concern individual models and their translations, we use model-based versions of the conditions $\mathsf{(E_{\Kmodel})}$ and $\mathsf{(G_{\Kmodel})}$ rather than the frame-based versions \textsf{(E)} and \textsf{(G)} introduced earlier.
We first define these model-based conditions and relate them to the corresponding frame-based notions, and then use them throughout the remainder of the section.

\begin{description}
    \item[{$\mathsf{(E_{\Kmodel})}$}] For any $w \in \Kmodel$ and $\varphi, \psi \in \MF$, $\Kmodel \models \varphi \leftrightarrow \psi \implies (\varphi \in \cl(G(w)) \iff \psi \in \cl(G(w)))$.
    \item[{$\mathsf{(G_{\Kmodel})}$}] For any $w \in \Kmodel$ and $\varphi \in \MF$, $\Kmodel \models \varphi \implies \varphi \in \cl(G(w))$.
\end{description}

Analogously to $\mathbf{CF}_\Lambda$, we write $\mathbf{CM}_{\Lambda}$ to indicate the class of closure models satisfying every condition \textsf{(X)} with $X \in \Lambda$ for $\Lambda \subseteq \{ \mathsf{E, E_{\Kmodel}, G, G_{\Kmodel}, M, C, N, K} \}$. 
Moreover, the operator $\logicof(\cdot)$ is now applied to the class of models: $\logicof(\mathbf{CM}_\Lambda) = \bigcap_{\Kmodel \in \mathbf{CM}_\Lambda} \logicof(\Kmodel)$.
In fact, it can be shown that, with respect to the formulas valid in the corresponding classes of models, $\mathsf{(E)}$ and $\mathsf{(E_{\Kmodel})}$, as well as $\mathsf{(G)}$ and $\mathsf{(G_{\Kmodel})}$, define the same modal logics.

\begin{proposition}
Let $\Lambda \subseteq \{ \mathsf{E, G, M, C, N, K} \}$ and $\Lambda' \subseteq \{ \mathsf{E_{\Kmodel}, G_{\Kmodel}, M, C, N, K} \}$ such that $\mathsf{E} \in \Lambda$ if and only if $\mathsf{E_{\Kmodel}} \in \Lambda'$, $\mathsf{G} \in \Lambda$ if and only if $\mathsf{G_{\Kmodel}} \in \Lambda'$, and $\Lambda \cap \{ \mathsf{M, C, N, K} \} = \Lambda' \cap \{ \mathsf{M, C, N, K} \}$. Then,
\[
    \logicof(\mathbf{CM}_{\Lambda}) =
    \logicof(\mathbf{CM}_{\Lambda'}) = L_{\Lambda}.
\]
\end{proposition}
\begin{proof}
By soundness and completeness, we have $\logicof(\mathrm{CM}_\Lambda) = L_\Lambda$.
Moreover, the canonical model constructed in Definition \ref{def:countermodel-full} satisfies $(\mathsf{E}_{\Kmodel})$ whenever $\mathsf{E} \in \Lambda$ and $(\mathsf{G}_{\Kmodel})$ whenever $\mathsf{G} \in \Lambda$ by the same argument as in the proof of Lemma \ref{lem:corresponding-lemma-full}.
Hence, if $\varphi\notin L_\Lambda$, this canonical model belongs to $\mathrm{CM}_{\Lambda'}$ and falsifies $\varphi$.
Together with soundness, this yields $\logicof(\mathrm{CM}_{\Lambda'})=L_\Lambda$.
\end{proof}

Thus, although the model-based conditions are in general stronger than their frame-based counterparts, replacing \textsf{(E)} and \textsf{(G)} with $(\mathsf{E}_{\Kmodel})$ and $(\mathsf{G}_{\Kmodel})$ does not change the resulting modal logics.
Since the remainder of this section concerns model-to-model representations, we henceforth write \textsf{(E)} and \textsf{(G)} for the model-based conditions $(\mathsf{E}_{\Kmodel})$ and $(\mathsf{G}_{\Kmodel})$, respectively.

\subsection{Closure Models and Neighborhood Models}
\label{subsec:closure-and-neighborhood}

\emph{Neighborhood semantics} interprets the modal operator by assigning to each world a collection of subsets of the set of possible worlds, called its \emph{neighborhoods}.
A formula $\Box\varphi$ is true at a world $w$ exactly when the truth set of $\varphi$ is one of the neighborhoods of $w$.
In this subsection, we compare neighborhood models with closure models by constructing truth-preserving translations in both directions.

\begin{definition}
    A \emph{neighborhood frame} is a tuple $\Kframe_N = (W, N)$ where $W$ is a non-empty set and $N$ is a function $N: W \to \powerset(\powerset(W))$. A \emph{neighborhood model} is a tuple $\Kmodel_N = (W, N, V)$ where $(W, N)$ is a neighborhood frame and $V$ is a valuation function $V: \Prop \to \powerset(W)$.
\end{definition}

For any model $\Kmodel$ and any formula $\varphi$, we write
\[
\truthset{\varphi}_{\Kmodel} = \{ w \in \Kmodel \mid  \Kmodel, w \models \varphi \}
\]
for the truth set of $\varphi$ in $\Kmodel$.
We use this notation throughout the remainder of the paper.

\begin{definition}
    The relation $\Kmodel_N, w \models \varphi$ is defined as for closure models, except for the modal operator, which is defined as follows:
    \[
        \Kmodel_N, w \models \Box \varphi \iff \truthset{\varphi}_{\Kmodel_N} \in N(w).
    \]
\end{definition}
The logic $\mathbf{E} = L_\mathsf{E}$ is sound and complete with respect to the class of neighborhood frames (see, e.g., \cite{chellas1980}).
$\mathbf{NM}$ denotes the class of all neighborhood models.

\begin{definition}
    Given a neighborhood model $\Kmodel_N = (W, N, V)$, we define $f_N(\Kmodel_N) = (W, G^{f_N}, \mathrm{id}, V)$ where
    \[
        G^{f_N}(w) = \{ \varphi \in \MF \mid \truthset{\varphi}_{\Kmodel_N} \in N(w) \}.
    \]
    \label{def:neighborhood-to-closure-function}
\end{definition}

\begin{lemma}
    Let $\Kmodel_N$ be a neighborhood model. Then, for any $w \in \Kmodel_N$ and $\varphi \in \MF$, $\Kmodel_N, w \models \varphi$ if and only if $f_N(\Kmodel_N), w \models \varphi$.
    \label{lem:neighborhood-to-closure-preserving}
\end{lemma}
\begin{proof}
    By induction on the complexity of $\varphi$. The atomic case is immediate because $f_N$ preserves the valuation. The Boolean cases are straightforward. Thus, it remains only to consider the case $\varphi = \Box \psi$, which is proven as follows:
    \begin{align*}
        \Kmodel_N, w \models \Box \psi &\iff \truthset{\psi}_{\Kmodel_N} \in N(w) \\
        &\iff \psi \in G^{f_N}(w) \\
        &\iff \psi \in \mathrm{id}(G^{f_N}(w)) \\
        &\iff f_N(\Kmodel_N), w \models \Box \psi. \qedhere
    \end{align*}
\end{proof}

\begin{lemma}
    For any neighborhood model $\Kmodel_N$, the model $f_N(\Kmodel_N)$ has the property \textsf{(E)}.
    \label{lem:neighborhood-has-E}
\end{lemma}
\begin{proof}
    Suppose that $f_N(\Kmodel_N) \models \varphi \leftrightarrow \psi$.
    \begin{align*}
        f_N(\Kmodel_N) \models \varphi \leftrightarrow \psi &\iff \truthset{\varphi}_{f_N(\Kmodel_N)} = \truthset{\psi}_{f_N(\Kmodel_N)} & \\
        &\iff \truthset{\varphi}_{\Kmodel_N} = \truthset{\psi}_{\Kmodel_N} &\text{Lemma \ref{lem:neighborhood-to-closure-preserving}} \\
        &\implies (\varphi \in G_{f_N}(w) \iff \psi \in G_{f_N}(w)) & \\
        &\iff (\varphi \in \mathrm{id}(G^{f_N}(w)) \iff \psi \in \mathrm{id}(G^{f_N}(w))). & \qedhere
    \end{align*}
\end{proof}

\begin{definition}
    Given a closure model $\Kmodel = (W, G, \cl, V)$, we define $g_N(\Kmodel) = (W, N^{g_N}, V)$ where
    \[
        N^{g_N}(w) = \{ \truthset{\varphi}_{\Kmodel} \mid \varphi \in \cl(G(w)) \}.
    \]
    \label{def:closure-to-neighborhood-function}
\end{definition}

\begin{lemma}
    Suppose that $\Kmodel$ satisfies \textsf{(E)}. Then, for any $w \in \Kmodel$ and $\varphi, \psi \in \MF$, $\truthset{\varphi}_{\Kmodel} = \truthset{\psi}_{\Kmodel}$ implies that $\varphi \in \cl(G(w))$ if and only if $\psi \in \cl(G(w))$.
    \label{lem:truthset-to-closure-by-E}
\end{lemma}
\begin{proof}
    If $\truthset{\varphi}_{\Kmodel} = \truthset{\psi}_{\Kmodel}$, then $w \models \varphi$ and $w \models \psi$ are equivalent for all $w \in \Kmodel$. This means that $w \models \varphi \leftrightarrow \psi$ for all $w$. Hence, $\Kmodel \models \varphi \leftrightarrow \psi$. Therefore, by \textsf{(E)}, $\varphi \in \cl(G(w))$ if and only if $\psi \in \cl(G(w))$.
\end{proof}

\begin{lemma}
    Let $\Kmodel$ be a closure model that satisfies \textsf{(E)}. Then, for any $w \in \Kmodel$ and $\varphi \in \MF$, $\Kmodel, w \models \varphi$ if and only if $g_N(\Kmodel), w \models \varphi$.
    \label{lem:closure-to-neighborhood-preserving}
\end{lemma}
\begin{proof}
    By induction on the complexity of $\varphi$.
    We only prove the case $\varphi = \Box \psi$.
    First, we show the left-to-right direction.
    If $\Kmodel, w \models \Box \psi$, then $\psi \in \cl(G(w))$.
    By definition of $N^{g_N}$, $\truthset{\psi}_{\Kmodel} \in N^{g_N}(w)$, which is equivalent to $\truthset{\psi}_{g_N(\Kmodel)} \in N^{g_N}(w)$ by the induction hypothesis. Therefore, we have $g_N(\Kmodel), w \models \Box \psi$.

    Conversely, suppose that $g_N(\Kmodel), w \models \Box \psi$. Then, $\truthset{\psi}_{g_N(\Kmodel)} \in N^{g_N}(w)$, which means $\truthset{\psi}_{\Kmodel} \in N^{g_N}(w)$ by the induction hypothesis. Thus, we have a formula $\chi$ such that $\chi \in \cl(G(w))$ and $\truthset{\psi}_{\Kmodel} = \truthset{\chi}_{\Kmodel}$. By Lemma \ref{lem:truthset-to-closure-by-E}, we have $\psi \in \cl(G(w))$, which means $\Kmodel, w \models \Box \psi$.
\end{proof}

From Lemmas \ref{lem:neighborhood-to-closure-preserving}, \ref{lem:neighborhood-has-E}, and \ref{lem:closure-to-neighborhood-preserving}, we now translate a neighborhood model into an equivalent closure model satisfying \textsf{(E)} and vice versa. Therefore, we get the following theorem.
\begin{theorem}
The translations $f_N: \mathbf{NM} \to \mathbf{CM}_\mathsf{E}$ and $g_N: \mathbf{CM}_\mathsf{E} \to \mathbf{NM}$ are truth-preserving.
\label{thm:neighborhood-closure-representation}
\end{theorem}
\begin{corollary}
    \[
        \logicof(\mathbf{NM}) = \logicof(\mathbf{CM_\mathsf{E}}).
    \]
    \label{cor:neighborhood-closure-equivalence}
\end{corollary}
Although this corollary also follows from soundness and completeness, it is obtained here directly from the truth-preserving translations.

\subsection{Closure Models and FMT Models}
\label{subsec:closure-and-FMT}

We next consider the relational semantics introduced by Fitting, Marek, and Truszczy\'{n}ski \cite{fitting1992} for the logic $\mathbf{N}$.
Unlike ordinary Kripke semantics, an FMT model assigns a binary relation $R_\varphi$ to each modal formula $\varphi$, and the truth of $\Box\varphi$ is determined by the $R_\varphi$-accessible worlds.
In this subsection, we compare FMT models with closure models satisfying \textsf{(G)} by constructing truth-preserving translations in both directions.

\begin{definition}
    An \emph{FMT frame} is a tuple $\Kframe_F = (W, \{R_\varphi\}_{\varphi \in \MF})$ where $W$ is a non-empty set and $R_\varphi \subseteq W \times W$ is a binary relation for each $\varphi \in \MF$.
    
    An \emph{FMT model} is a tuple $\Kmodel_F = (W, \{R_\varphi\}_{\varphi \in \MF}, V)$ where $(W, \{R_\varphi\}_{\varphi \in \MF})$ is an FMT frame and $V$ is a valuation function $V: \Prop \to \powerset(W)$.
    \label{def:FMT model}
\end{definition}

\begin{definition}
    The relation $\Kmodel_F, w \models \varphi$ is defined as for closure models, except for the modal operator, which is defined as follows:
    \[
        \Kmodel_F, w \models \Box \varphi
        \iff
        \text{for every } v \in W,\ w R_\varphi v \text{ implies } \Kmodel_F, v \models \varphi.
    \]
    \label{def:FMT-satisfaction}
\end{definition}
As shown in \cite{fitting1992} (see also \cite{kurahashi2026N} for a recent presentation), the logic $\mathbf{N} = L_{\mathsf{G}}$ is sound and complete with respect to $\mathbf{FM}$, which denotes the class of all FMT models.

\begin{definition}
    Given an FMT model $\Kmodel_F = (W, \{R_\varphi\}_{\varphi \in \MF}, V)$, we define $f_F(\Kmodel_F) = (W, G^{f_F}, \mathrm{id}, V)$ where
    \[
        G^{f_F}(w)
        =
        \{ \varphi \in \MF \mid
        R_\varphi[w]
        \subseteq \truthset{\varphi}_{\Kmodel_F} \}.
    \]
    Here, $R_\varphi[w]$ denotes the set $\{v \in W \mid w R_\varphi v\}$.
    \label{def:FMT-to-closure-function}
\end{definition}

Equivalently,
\[
    \varphi \in G^{f_F}(w) \iff (w R_\varphi v \implies \Kmodel_F, v \models \varphi) \text{ for every } v \in W.  
\]

\begin{lemma}
    Let $\Kmodel_F$ be an FMT model. Then, for any $w \in \Kmodel_F$ and $\varphi \in \MF$, $\Kmodel_F, w \models \varphi$ if and only if $f_F(\Kmodel_F), w \models \varphi$.
    \label{lem:FMT-to-closure-preserving}
\end{lemma}
\begin{proof}
    By induction on the complexity of $\varphi$.
    The atomic case and the Boolean cases are straightforward, so we consider only the modal case $\varphi = \Box \psi$:
    \begin{align*}
        \Kmodel_F, w \models \Box \psi
        &\iff  (w R_\psi v \implies \Kmodel_F, v \models \psi) \text{ for all } v \in W \\
        &\iff \psi \in G^{f_F}(w) \\
        &\iff \psi \in \mathrm{id}(G^{f_F}(w)) \\
        &\iff f_F(\Kmodel_F), w \models \Box \psi. \qedhere
    \end{align*}
\end{proof}

\begin{lemma}
    For any FMT model $\Kmodel_F$, the model $f_F(\Kmodel_F)$ has the property \textsf{(G)}.
    \label{lem:FMT-has-G}
\end{lemma}
\begin{proof}
    Suppose that $f_F(\Kmodel_F) \models \varphi$. By Lemma \ref{lem:FMT-to-closure-preserving}, $\Kmodel_F \models \varphi$, and hence $\truthset{\varphi}_{\Kmodel_F}=W$. Then, for every $w \in W$,
    \[
        \{v \in W \mid wR_\varphi v\}
        \subseteq
        \truthset{\varphi}_{\Kmodel_F}.
    \]
    Thus, $\varphi \in G^{f_F}(w)=\mathrm{id}(G^{f_F}(w))$ for every $w \in W$. Therefore, $f_F(\Kmodel_F)$ has the property \textsf{(G)}.
\end{proof}

Next, we provide the translation in the opposite direction.
The idea behind this definition of $R_\varphi$ is to preserve the truth value of $\Box\varphi$.
If $\varphi \notin \cl(G(w))$, and hence $\Box\varphi$ is false at $w$ in the original closure model, we define $R_\varphi$ so that there is an $R_\varphi$-accessible world at which $\varphi$ is false, thereby providing a witness to the failure of $\Box\varphi$ in the translated FMT model.
On the other hand, if $\varphi \in \cl(G(w))$, no such witness is introduced, so that $\Box\varphi$ remains true after the translation.

\begin{definition}
    Given a closure model $\Kmodel = (W, G, \cl, V)$, we define $g_F(\Kmodel) = (W, \{R^{g_F}_\varphi\}_{\varphi \in \MF}, V)$ where:
    \[
    R_\varphi^{g_F}[w] =
        \begin{cases}
            W \setminus \truthset{\varphi}_{\Kmodel} &\text{if } \varphi \notin \cl(G(w)) \\
            \emptyset &\text{otherwise.}
        \end{cases}
    \]
    \label{def:closure-to-FMT-function}
\end{definition}
Equivalently,
\[
    w R^{g_F}_\varphi v \iff \varphi \notin \cl(G(w)) \text{ and } \Kmodel, v \not\models \varphi.
\]

\begin{lemma}
    Let $\Kmodel$ be a closure model that satisfies \textsf{(G)}. Then, for any $w \in \Kmodel$ and $\varphi \in \MF$, $\Kmodel, w \models \varphi$ if and only if $g_F(\Kmodel), w \models \varphi$.
    \label{lem:closure-to-FMT-preserving}
\end{lemma}
\begin{proof}
    By induction on the complexity of $\varphi$.
    The base case and the Boolean cases are straightforward.
    We only prove the case $\varphi = \Box \psi$.

    Suppose first that $\Kmodel, w \models \Box \psi$. Then $\psi \in \cl(G(w))$. By the definition of $R^{g_F}_\psi$, there is no $v \in W$ such that $wR^{g_F}_\psi v$. Hence, $g_F(\Kmodel), w \models \Box \psi$.

    Conversely, suppose that $\Kmodel, w \notmodels \Box \psi$. Then, $\psi \notin \cl(G(w))$. By the contrapositive of \textsf{(G)}, there is some $v \in W$ such that $\Kmodel, v \notmodels \psi$. By the definition of $R^{g_F}_\psi$, we have $wR^{g_F}_\psi v$. Moreover, by the induction hypothesis, $\Kmodel, v \notmodels \psi$ implies $g_F(\Kmodel), v \notmodels \psi$. Therefore, $g_F(\Kmodel), w \notmodels \Box \psi$.
\end{proof}

Lemmas \ref{lem:FMT-to-closure-preserving}, \ref{lem:FMT-has-G}, and \ref{lem:closure-to-FMT-preserving} yield the following theorem.
\begin{theorem}
The translations $f_F: \mathbf{FM} \to \mathbf{CM}_\mathsf{G}$ and $g_F: \mathbf{CM}_\mathsf{G} \to \mathbf{FM}$ are truth-preserving.
\label{thm:FMT-closure-representation}
\end{theorem}
\begin{corollary}
    \[
        \logicof(\mathbf{FM}) = \logicof(\mathbf{CM_\mathsf{G}}).
    \]
    \label{cor:FMT-closure-equivalence}
\end{corollary}

\subsection{Closure Models and Kripke Models}
\label{subsec:closure-and-Kripke}

We next compare closure models with ordinary Kripke models.
As we will see, the properties \textsf{(G)} and \textsf{(K)} alone do not suffice to obtain a direct representation of every closure model by a Kripke model.
We therefore introduce an additional condition expressing Kripke representability.

We begin by briefly recalling the standard Kripke semantics.
\begin{definition}
    A \emph{Kripke frame} is a tuple $\Kframe_K= (W, R)$ where $W$ is a non-empty set and $R \subseteq W \times W$. A \emph{Kripke model} is a tuple $\Kmodel_K = (W, R, V)$ where $(W, R)$ is a Kripke frame and $V$ is a valuation function $V: \Prop \to \powerset(W)$.
    The relation $\Kmodel_K, w \models \varphi$ is defined in the usual way, and the modal operator is interpreted as follows:
    \[
        \Kmodel_K, w \models \Box \varphi
        \iff
        \text{for all } v \in W, \ w R v\text{ implies } \Kmodel_K, v \models \varphi.
    \]
    \label{def:Kripke-model}
\end{definition}

It is well known (see \cite{blackburn2001, chellas1980, hughes1996}) that the logic $\mathbf{K} = L_\mathsf{GK}$ is sound and complete with respect to $\mathbf{KM}$, the class of all Kripke models.

\subsubsection{Kripke Representability}

\begin{definition}
    We say that a closure model $\Kmodel=(W, G, \cl, V)$ has the property \textsf{(KRep)} if, for every $w\in W$ and $\varphi \in \MF$,
    \[
        \varphi\in\cl(G(w))
        \iff
        \bigcap_{\psi\in\cl(G(w))}\truthset{\psi}_{\Kmodel}
        \subseteq
        \truthset{\varphi}_{\Kmodel}.
    \]
    \label{def:KRep}
\end{definition}
If a closure model satisfies \textsf{(KRep)}, we call it \emph{Kripke-representable}. We write $\mathbf{CM}_\mathsf{KRep}$ to mean the class of Kripke-representable closure models.

The left-to-right direction in \textsf{(KRep)} follows immediately from the definition of intersection. Thus, the substantial part of \textsf{(KRep)} is the converse direction. Equivalently, \textsf{(KRep)} can be written in the following separation form:
\begin{align*}
    \varphi\notin\cl(G(w))
    \implies
    &\text{there is some }v\in W\text{ such that } \\
    &\Kmodel, v \models \psi \text{ for every }\psi\in\cl(G(w)) \text{ and }\Kmodel,v\not\models\varphi.
\end{align*}
Thus, \textsf{(KRep)} says that every formula outside $\cl(G(w))$ can be separated from $\cl(G(w))$ by a world of the model.

There is also a useful semantic characterization of \textsf{(KRep)}.
For a closure model $\Kmodel$, define the model-relative consequence operator $C_{\Kmodel}$ by
\[
    C_{\Kmodel}(\Gamma) = \{ \varphi \in \MF \mid \Kmodel, v \models \Gamma \implies \Kmodel, v \models \varphi \text{ for every } v \in W \}.
\]
Then, \textsf{(KRep)} is equivalent to $\cl(G(w)) = C_{\Kmodel}(\cl(G(w)))$ for every $w\in W$. 
This characterization is closely aligned with ordinary Kripke semantics.
Let $\Kmodel_K = (W, R, V)$ be a Kripke model, and define
\[
    H(w) = \{ \varphi \mid R[w] \subseteq \truthset{\varphi}_{\Kmodel_K} \}.
\]
Thus, $H(w)$ is the set of formulas true at every world accessible from $w$.
If $C_{\Kmodel_K}$ is defined for $\Kmodel_K$ in the same way as above, then $H(w) = C_{\Kmodel_K}(H(w))$.
Indeed, every world in $R[w]$ satisfies all formulas in $H(w)$, so every semantic consequence of $H(w)$ in $\Kmodel_K$ is also true throughout $R[w]$ and hence belongs to $H(w)$.

The condition \textsf{(KRep)} is strictly stronger than the combination of \textsf{(G)} and \textsf{(K)}.
We first show that every Kripke-representable closure model satisfies both \textsf{(G)} and \textsf{(K)}, and then give a counterexample to the converse.
\begin{lemma}
    Every closure model satisfying \textsf{(KRep)} satisfies both \textsf{(G)} and \textsf{(K)}.
    \label{lem:KRep-implies-GK}
\end{lemma}
\begin{proof}
    Suppose first that $\Kmodel\models\varphi$. Then $\truthset{\varphi}_{\Kmodel}=W$. Hence, for every $w\in W$,
    \[
        \bigcap_{\psi\in\cl(G(w))}\truthset{\psi}_{\Kmodel}
        \subseteq
        W
        =
        \truthset{\varphi}_{\Kmodel}.
    \]
    By \textsf{(KRep)}, $\varphi\in\cl(G(w))$. Therefore, $\Kmodel$ satisfies \textsf{(G)}.

    Next, suppose that $\varphi\to\psi\in\cl(G(w))$ and $\varphi\in\cl(G(w))$. Then
    \[
        \bigcap_{\chi\in\cl(G(w))}\truthset{\chi}_{\Kmodel}
        \subseteq
        \truthset{\varphi\to\psi}_{\Kmodel}
        \cap
        \truthset{\varphi}_{\Kmodel}
        \subseteq
        \truthset{\psi}_{\Kmodel}.
    \]
    By \textsf{(KRep)}, $\psi\in\cl(G(w))$. Therefore, $\Kmodel$ satisfies \textsf{(K)}.
\end{proof}

The converse of Lemma \ref{lem:KRep-implies-GK} does not hold. The following example shows that the combination of \textsf{(G)} and \textsf{(K)} does not imply \textsf{(KRep)}.

\begin{example}
    Let $\Kmodel = (W, G, \cl, V)$ be a closure model where:
    \begin{align*}
        W &= \mathbb{N} \\
        G(w) &= \{ \varphi \mid W \setminus \truthset{\varphi}_{\Kmodel} \text{ is finite} \} \\
        \cl &= \mathrm{id} \\
        V(p_n) &= W \setminus \{n\} \text{ for all } n \in \mathbb{N},
        V(q) = \{ 2n \mid n \in \mathbb{N} \}.
    \end{align*}
    Note that this definition is well-founded, although $G$ is defined in terms of truth sets in $\Kmodel$.
    Determining whether $\varphi \in G(w)$ requires only the truth set of $\varphi$, while this membership is used only to determine the truth value of the strictly more complex formula $\Box \varphi$.
    
    This model satisfies $\textsf{(G)}$.
    If $\Kmodel \models \varphi$, we have $\truthset{\varphi}_{\Kmodel} = W$, which is equivalent to $W \setminus \truthset{\varphi}_{\Kmodel} = \emptyset$.
    Therefore, $\varphi \in \cl(G(w))$.
    It also satisfies $\textsf{(K)}$. Suppose that $\varphi, \varphi \to \psi \in \cl(G(w))$. Then both $\truthset{\varphi\to\psi}_{\Kmodel}$ and $\truthset{\varphi}_{\Kmodel}$ have finite complements.
    Since $\truthset{\varphi\to\psi}_{\Kmodel} \cap \truthset{\varphi}_{\Kmodel} \subseteq \truthset{\psi}_{\Kmodel}$, the set $\truthset{\psi}_{\Kmodel}$ also has a finite complement. Therefore, $\psi \in \cl(G(w))$.

    On the other hand, $p_n \in \cl(G(w))$ for every $n \in \mathbb{N}$, since $V(p_n)=W\setminus\{n\}$ is cofinite. Therefore,
    \[
        \bigcap_{\varphi\in \cl(G(w))}\truthset{\varphi}_{\Kmodel}
        \subseteq
        \bigcap_{n\in\mathbb{N}}\truthset{p_n}_{\Kmodel}
        =
        \emptyset.
    \]
    Since $\truthset{q}_{\Kmodel} = \{ 2n \mid n \in \mathbb{N} \}$, we have
    \[
        \bigcap_{\varphi \in \cl(G(w))} \truthset{\varphi}_{\Kmodel}
        \subseteq
        \truthset{q}_{\Kmodel}.
    \]
    However, $q \notin G(w)$ because $W \setminus \truthset{q}_{\Kmodel}$ is not finite. Therefore, the model does not satisfy \textsf{(KRep)}.
    \label{ex:GK-not-KRep}
\end{example}

\subsubsection{Representations between Kripke Models and Kripke-Representable Closure Models}

In this subsection, we construct truth-preserving translations between Kripke models and Kripke-representable closure models.
These translations show that the valid formulas in all Kripke models coincide with those valid in all Kripke-representable closure models.
Together with the completeness results established in Section \ref{sec:logic}, this further shows that they coincide with the formulas valid in all closure models satisfying \textsf{(G)} and \textsf{(K)}.

\begin{definition}
    Given a Kripke model $\Kmodel_K = (W, R, V)$, we define $f_K(\Kmodel_K) = (W, G^{f_K}, \mathrm{id}, V)$ where
    \[
        G^{f_K}(w)
        =
        \{\varphi\in\MF\mid
        R[w]
        \subseteq
        \truthset{\varphi}_{\Kmodel_K}\}.
    \]
    \label{def:Kripke-to-closure-function}
\end{definition}
As in the case for $f_F$, we observe that $\varphi \in G^{f_K}(w)$ if and only if, for all $v \in W$, $w R v$ implies $\Kmodel_K, v \models \varphi$.

\begin{lemma}
    Let $\Kmodel_K$ be a Kripke model. Then, for any $w\in\Kmodel_K$ and $\varphi\in\MF$, $\Kmodel_K,w\models\varphi$ if and only if $f_K(\Kmodel_K),w\models\varphi$.
    \label{lem:Kripke-to-closure-preserving}
\end{lemma}
\begin{proof}
    By induction on the complexity of $\varphi$. The base case and the Boolean cases are straightforward. We only prove the case $\varphi=\Box\psi$ as follows:
    \begin{align*}
        \Kmodel_K,w\models\Box\psi
        &\iff (w R v \implies \Kmodel_K, v \models \psi) \text{ for all } v \in \Kmodel_K \\
        &\iff \psi\in G^{f_K}(w) \\
        &\iff \psi\in\mathrm{id}(G^{f_K}(w)) \\
        &\iff f_K(\Kmodel_K),w\models\Box\psi. \qedhere
    \end{align*}
\end{proof}

\begin{lemma}
    For any Kripke model $\Kmodel_K$, the closure model $f_K(\Kmodel_K)$ has the property \textsf{(KRep)}.
    \label{lem:Kripke-has-KRep}
\end{lemma}
\begin{proof}
    Let $w \in W$ and $\varphi \in \MF$. By Lemma \ref{lem:Kripke-to-closure-preserving}, the truth sets of formulas in $\Kmodel_K$ and $f_K(\Kmodel_K)$ coincide, i.e., we can freely use the fact that $\truthset{\varphi}_{\Kmodel_K} = \truthset{\varphi}_{f_K(\Kmodel_K)}$ for any $\varphi$.

    The left-to-right direction is straightforward. For the converse direction, suppose that
    \[
        \bigcap_{\psi\in\mathrm{id}(G^{f_K}(w))} \truthset{\psi}_{f_K(\Kmodel_K)} \subseteq \truthset{\varphi}_{f_K(\Kmodel_K)} .
    \]
    For every $\psi\in G^{f_K}(w)$, we have $R[w]\subseteq\truthset{\psi}_{\Kmodel_K}$. Hence,
    \[
        R[w]
        \subseteq
        \bigcap_{\psi\in G^{f_K}(w)}\truthset{\psi}_{\Kmodel_K}
        =
        \bigcap_{\psi\in\mathrm{id}(G^{f_K}(w))}\truthset{\psi}_{f_K(\Kmodel_K)}.
    \]
    Then, $R[w]\subseteq\truthset{\varphi}_{\Kmodel_K}$, and thus $\varphi\in G^{f_K}(w)=\mathrm{id}(G^{f_K}(w))$. Therefore, $f_K(\Kmodel_K)$ satisfies \textsf{(KRep)}.
\end{proof}

\begin{definition}
    Given a closure model $\Kmodel=(W,G,\cl,V)$, we define $g_K(\Kmodel)=(W,R^{g_K},V)$ where
    \[
        R^{g_K}[w]
        =        \bigcap_{\varphi\in\cl(G(w))}\truthset{\varphi}_{\Kmodel}.
    \]
    \label{def:closure-to-Kripke-function}
\end{definition}
Equivalently,
\[
    wR^{g_K}v
    \iff
    \Kmodel,v\models\varphi
    \text{ for every }\varphi\in\cl(G(w)).
\]

\begin{lemma}
    Let $\Kmodel$ be a closure model that satisfies \textsf{(KRep)}. Then, for any $w\in\Kmodel$ and $\varphi\in\MF$, $\Kmodel,w\models\varphi$ if and only if $g_K(\Kmodel),w\models\varphi$.
    \label{lem:closure-to-Kripke-preserving}
\end{lemma}
\begin{proof}
    By induction on the complexity of $\varphi$. The base case and the Boolean cases are straightforward. We only prove the case $\varphi=\Box\psi$.

    For the left-to-right direction, suppose $\Kmodel, w \models \Box \psi$ and take an arbitrary $v$ such that $w R^{g_K} v$. From the former, $\psi \in \cl(G(w))$. From the definition of $R^{g_K}$, it follows that $\Kmodel, v \models \psi$, which implies $g_K(\Kmodel), v \models \psi$ by the induction hypothesis. Therefore, $g_K(\Kmodel), w \models \Box \psi$.

    The other direction is shown by using the contraposition. Suppose that $\Kmodel, w \notmodels \Box \psi$. Then, $\psi \notin \cl(G(w))$. From \textsf{(KRep)}, we have some $v \in \Kmodel$ such that $\Kmodel, v \models \chi$ for all $\chi \in \cl(G(w))$ and $\Kmodel, v \notmodels \psi$. Take such $v$, and we have $w R^{g_K} v$ and $\Kmodel, v \notmodels \psi$. From the latter together with the induction hypothesis, $g_K(\Kmodel), v \notmodels \psi$ holds. Therefore, we have $g_K(\Kmodel), w \notmodels \Box \psi$.
\end{proof}

Now, we show the correspondence between Kripke-representable models and Kripke models:
\begin{theorem}
    The translations $f_K: \mathbf{KM} \to \mathbf{CM}_\mathsf{KRep}$ and $g_K: \mathbf{CM}_\mathsf{KRep} \to \mathbf{KM}$ are truth-preserving.
    \label{thm:Kripke-closure-representation}
\end{theorem}

\begin{corollary}
    \[
        \logicof(\mathbf{KM}) =
        \logicof(\mathbf{CM_\mathsf{KRep}}) = \logicof(\mathbf{CM_\mathsf{GK}}).
    \]
    \label{thm:Kripke-closure-equivalence}
\end{corollary}
\vspace{-2.0em}
\begin{proof}
    The equality of the first two follows from the theorem above. By the soundness and completeness theorem for $L_{\mathsf{GK}}$ with respect to $\mathbf{CM_\mathsf{GK}}$ and Kripke models,  we have
    \[
        \logicof(\mathbf{CM_\mathsf{GK}})
        = \logicof(\mathbf{KM})
        = L_{\mathsf{GK}}. \qedhere
    \]
\end{proof} 
It is well known that the normal modal logic $\mathbf{K}$ can equivalently be axiomatized by the rule $\mathsf{RE}$ and axiom schemata $\mathsf{M}$, $\mathsf{C}$, and $\mathsf{N}$, so that $L_{GK}=L_{EMCN}$.
Therefore, by soundness and completeness,
we have $\logicof(\mathbf{CM}_\mathsf{KRep}) = \logicof(\mathbf{CM}_\mathsf{GK}) = \logicof(\mathbf{CM}_\mathsf{EMCN})$.

% --- 
Although \textsf{(KRep)} is a genuinely stronger model-theoretic condition than \textsf{(G)} plus \textsf{(K)}, they determine the same modal logic. This is analogous to the distinction between neighborhood models satisfying the properties corresponding to the axiom schemata \textsf{M}, \textsf{C}, \textsf{N}, and augmented neighborhood models: finite intersection and supplementation do not, in general, guarantee that a neighborhood is generated by a single core.

\section{Prospects}
\label{sec:prospects}

In this paper, we introduced closure semantics as a new semantic framework for weak modal logics.
As a basic result, we established completeness for a range of weak modal logics, including the least modal logic $\mathbf{ML}_0$.
We also gave translations between closure models and three existing possible-world semantics, together with conditions under which such translations are available.
These results, however, constitute only a first step toward a systematic study of closure semantics.

A first direction is to extend the completeness results to a wider range of modal principles.
For example, one may consider the familiar axioms $\mathsf{D}$, $\mathsf{T}$, $\mathsf{4}$, $\mathsf{B}$, and $\mathsf{5}$ from normal modal logic.
Beyond treating such axioms individually, it may be possible to develop a more general completeness theory.
In the completeness proofs for logics containing $\mathsf{C}$, $\mathsf{N}$, and $\mathsf{K}$, the canonical closure $\cl_\Lambda$ is generated by rules of a Horn-like form.
This suggests that, whenever the closure conditions corresponding to modal axioms can be expressed by a suitable class of Horn-type rules, completeness might be established uniformly by constructing the least closure satisfying those rules.

A second direction concerns definability and correspondence for closure semantics.
We have room to investigate which properties of closure models are definable by modal formulas, and conversely, which modal formulas correspond to natural conditions on closure models.
For normal modal logic, these questions have led to fundamental results such as the Goldblatt-Thomason theorem from a model-theoretic perspective and Sahlqvist theory from a syntactic perspective \cite{blackburn2001}.
It is therefore natural to ask whether analogous theories can be developed for closure semantics.

Closure semantics may provide a particularly flexible setting for such investigations.
The assignment $G(w)$ can be specified quite freely, independently of the truth of formulas at $w$, and only weak assumptions are imposed on $\cl$ itself.
This leaves considerable room for formulating structural conditions on closure frames and suggests that a wide variety of modal principles may admit corresponding closure conditions.
At the same time, however, it is necessary to distinguish carefully between conditions on closure frames and those on closure models.
For example, a direct condition corresponding to the axiom $\mathsf{T}: \Box \varphi \to \varphi$, would be
\[
    \varphi \in \cl(G(w)) \implies w \models \varphi.
\]
Since this condition refers explicitly to truth in a model, it is not itself a condition on closure frames.

A model theory for closure models also remains to be developed.
For Kripke models, structural operations on models, morphisms such as bounded morphisms, and bisimulation invariance form a well-developed theory.
It is natural to ask which of these ideas have meaningful counterparts in closure semantics.
For example, a map preserving modal truth would presumably require more than agreement on the truth values of propositional variables at corresponding worlds; it would also have to respect the associated generators $G$ and the closure operation $\cl$.
This is analogous to the case of $p$-morphisms, whose conditions concern not only truth at worlds but also the structure of the accessible worlds.

Questions concerning the finite model property, decidability, and computational complexity also remain open for broader classes of logics in closure semantics.
In the present paper, finite model constructions and primitive-recursive decidability were obtained only for logics without modal inference rules, and it is natural to ask how far these results can be extended.
A further problem is determining the complexity of validity and satisfiability in logics characterized by closure models.
Since a finite generator $G(w)$ may determine an infinite set $\cl(G(w))$, such complexity analyses may have to take into account not only the size of finite models but also the computational cost of the closure operation itself.

Finally, a more fundamental and difficult question remains: how should closure models themselves be interpreted?
Kripke semantics is accompanied by an intuitive interpretation of its mathematical structure: worlds are understood as possible alternatives, while accessibility represents a relation among them.
Closure semantics, by contrast, presently provides a formal truth definition without a comparably general interpretation of the roles played by the generator $G(w)$ and the operation $\cl$.
However, it may be too early to address this question directly.
The structural and computational results obtained along the way may eventually clarify what closure models represent.
In turn, this may pave the way toward a more intuitive understanding of closure semantics.

\section*{Acknowledgments}
I would like to thank Taishi Kurahashi and Mashu Noguchi for helpful discussions that greatly contributed to the development of this paper.

This research was supported by OpenAI through access to ChatGPT for Academic Researchers, which was used to assist with proofs, related arguments, and the preparation of the manuscript. All contents were independently checked by the author, who takes full responsibility for the final manuscript.
\printbibliography
\end{document}